\documentclass[11pt, reqno]{amsart}
\usepackage{graphicx, amssymb, amsmath, amsthm}
\usepackage[utf8]{inputenc}
\usepackage{epsfig}
\usepackage{hyperref}
\usepackage{comment}
\usepackage{mathrsfs}
\numberwithin{equation}{section}
\usepackage{amsthm}
\usepackage{subfigure}
\usepackage{tikz}
\usepackage{here}
\usepackage{float}
\usepackage{pifont}
\usepackage{enumitem}
\usetikzlibrary{matrix,arrows}
\usetikzlibrary{shapes}
\usetikzlibrary{calc}
\usetikzlibrary{arrows}
\usetikzlibrary{decorations.pathreplacing,decorations.markings}

\usepackage{tikz-cd}

\usetikzlibrary{patterns}

\newtheorem{theorem}{Theorem}[section]

\newtheorem{lemma}[theorem]{Lemma}

\newtheorem{conjecture}[theorem]{Conjecture}
\newtheorem{proposition}[theorem]{Proposition}
\newtheorem{question}[theorem]{Question}

\theoremstyle{definition}

\makeatletter
\newcommand{\Extend}[5]{\ext@arrow0099{\arrowfill@#1#2#3}{#4}{#5}}
\makeatother

\DeclareMathOperator{\Lip}{Lip}
\DeclareMathOperator{\dist}{dist}

\DeclareMathOperator{\vol}{vol}
\DeclareMathOperator{\spt}{spt}

\begin{document}

\title[Calabi-Yau type theorem]{Calabi-Yau type theorem for complete manifolds with nonnegative scalar curvature}

\author[J. Zhu]{Jintian Zhu}
\address[Jintian Zhu]{Institute for Theoretical Sciences, Westlake University, 600 Dunyu Road, 310030, Hangzhou, Zhejiang, People's Republic of China}
\email{zhujintian@westlake.edu.cn}

\begin{abstract}
In this paper, we are able to prove an analogy of the Calabi-Yau theorem for complete Riemannian manifolds with nonnegative scalar curvature which are aspherical at infinity. The key tool is an existence result for arbitrarily large {\it bounded} regions with weakly mean-concave boundary in Riemannian manifolds with sublinear volume growth.
As an application, we use the same tool to show that a complete contractible Riemannian $3$-manifold with positive scalar curvature and sublinear volume growth is necessarily homeomorphic to $\mathbb R^3$.
\end{abstract}

\maketitle

\section{Introduction}

Let $(M,g)$ be a complete Riemannian manifold. Recall that $(M,g)$ has {\it sublinear volume growth} if it holds
$$\liminf_{r\to +\infty}\frac{\vol_g(B_r(p))}{r}=0\mbox{ for some point }p\in M,$$
where the choice of the point $p$ plays no role in this definition. It is well-known that there are interplays between nonnegative curvatures and the underlying topology on complete Riemannian manifolds with sublinear volume growth. The famous Calabi-Yau theorem \cite{Calabi75, Yau76} states that if $(M,g)$ has nonnegative Ricci curvature, then it has sublinear volume growth if and only if it is closed. Very recently, such purely geometric way of characterizing compactness for complete manifolds under nonnegative-curvature condition plays a crucial role in the author's establishment of various rigidity results on complete manifolds with nonnegative scalar curvature (see the author's previous works \cite{Zhu23a, Zhu23+}).

The goal of this paper is to establish an analogy of the Calabi-Yau theorem with much weaker condition of nonnegative scalar curvature.  First we point out that additional conditions are necessary in order to obtain such kind of result  since there are indeed open complete Riemannian manifolds with nonnegative scalar curvature even having finite volume (see Appendix \ref{sec: finite volume}).

As a natural choice, we consider the class of manifolds which are {\it aspherical at infinity}. 
Recall that a manifold $M$ is said to be aspherical at infinity means that for each positive integer $i\geq 2$ and compact subset $K\subset M$ there is a larger compact subset $\tilde K\supset K$ such that the inclusion map $\pi_i(M\setminus \tilde K)\to \pi_i(M\setminus K)$ is the zero map. 
Our main theorem is
\begin{theorem}\label{Thm: main}
Let $(M^n,g)$, $2\leq n\leq 4$, be a complete Riemannian manifold with nonnegative scalar curvature outside a compact subset, which is aspherical at infinity. Then it has sublinear volume growth if and only if it is closed.
\end{theorem}
We emphasize that the extra aspherical condition at infinity is natural in the sense that every open surface appears to be aspherical at infinity. In other words, Theorem \ref{Thm: main} is indeed a generalization of the 2D Calabi-Yau theorem in the context of scalar curvature. 

Given the Calabi-Yau type Theorem \ref{Thm: main} and also the aspherical splitting theorem established in \cite[Theorem 1.7]{HZ23+}, there may be an underlying principle saying that {\it nonnegative scalar curvature on complete open manifolds with extra natural topological condition has similar behaviours with nonnegative Ricci curvature on complete open manifolds}. Here natural means that open surfaces automatically satisfy these topological conditions.

The main strategy to prove Theorem \ref{Thm: main} is the following existence lemma for an arbitrarily large {\it bounded} region with weakly mean-concave boundary in complete open Riemannian manifolds with sublinear volume growth. The proof is based on the $\mu$-bubble method raised by Gromov \cite{Gromov23}. Recently, Gromov's $\mu$-bubble method has led to applications in many other geometric problems (refer to \cite{Gromov20+, CL20+, LUY21+, Zhu21, CL23, CRZ23, Zhu23b, CLMS24+} and references therein). 
\begin{lemma}\label{Lem: concave}
    Let $(M^n,g)$, $2\leq n\leq 7$, be an open complete Riemannian manifold with sublinear volume growth. Then for any bounded region $K$ we can find a larger bounded region $\hat\Omega\supset K$ such that $\partial\Omega$ is weakly mean-concave with respect to the unit outer normal of $\partial\Omega$ in $\Omega$.
\end{lemma}

We note that it follows from the previous works \cite{CL20, Song23} that there is at least one embedded minimal hypersurface in a complete Riemannian manifold with sublinear volume growth, but the compactness of the minimal hypersurface cannot be guaranteed therein. In this work, the pass from a minimal hypersurface to a $\mu$-bubble successfully handle the compactness issue. 

With Lemma \ref{Lem: concave} we are able to provide an alternative proof of the Calabi-Yau theorem combined with the splitting theorem \cite[Theorem 2]{CK92}, which has the obvious advantage that the Bishop-Gromov volume comparison need not be used, and this explains the reason why our method can be generalized to the weaker setting of scalar curvature. With the same philosophy, we just reduce Theorem \ref{Thm: main} to the validity of the following splitting theorem:

\begin{proposition}\label{Prop: splitting}
    Let $(M^n_\infty,g)$, $3\leq n\leq 7$, be a complete and non-compact Riemannian manifold with compact and weakly mean-convex boundary $\partial M_\infty$. If $(M_\infty,g)$ has nonnegative scalar curvature and any embedded hypersurface representing a non-trivial $\mathbb Q$-homology class cannot admit any metric with positive scalar curvature, then $(M_\infty,g)$ must split into the Riemannian product $(\partial M_\infty,g|_{\partial M_\infty})\times [0,+\infty)$.
\end{proposition}

In the same spirit of the Calabi-Yau theorem, we also mention the Gromov-Lawson theorem below \cite[Theorem 8.11]{GL83}.
\begin{theorem}[Gromov-Lawson]
    Let $\Sigma$ be a complete non-compact Riemannian surface with the property that, for some fixed constant $\alpha>1/2$,
    $$I_\alpha(\varphi):=\int_\Sigma|\nabla\varphi|^2+\alpha K\varphi^2\,\mathrm d\sigma\geq 0\mbox{ for all }\varphi\in C_0^\infty(\Sigma),$$
    where $K$ is the Gaussian curvature of $\Sigma$.
    Then $\Sigma$ has infinite volume.
\end{theorem}
We remark that the nonnegative curvature condition is weakened from the pointwise sense to some spectrum sense in above Gromov-Lawson theorem, which seems even not to be well understood in the context of Ricci curvature. Concerning our Theorem \ref{Thm: main} it is very natural to ask the following
\begin{question}
    Let $\alpha$ be a positive constant greater than $1/4$. Let $(M^n,g)$, $n\geq 3$, be a complete Riemannian manifold with
    $$
    \int_{M}|\nabla_g\varphi|^2+\alpha R(g)\varphi^2\,\mathrm d\mu_g\geq 0\mbox{ for all }\varphi\neq 0\in C_0^\infty(M),
    $$
    which is aspherical at infinity. If $(M,g)$ has finite volume, does it have to be compact?
\end{question}

For further application of Lemma \ref{Lem: concave} we mention the following 
\begin{question}\label{Q: contractible}
Is any complete contractible $3$-manifold with positive scalar curvature necessarily homeomorphic to $\mathbb R^3$?
\end{question}
This question was considered by Wang in his works \cite{Wang19+,Wang23+}, where he proved that if $(M,g)$ is a complete contractible $3$-manifold with positive scalar curvature, then its fundamental group at infinity must be trivial. As a special case of Question \ref{Q: contractible}, we are able to prove
\begin{theorem}\label{Thm: topology}
    Let $(M^3,g)$ be a complete and contractible Riemannian $3$-manifold with positive scalar curvature and sublinear volume growth. Then $M$ is homeomorphic to $\mathbb R^3$.
\end{theorem}

\section*{Acknowledgement}
The author is grateful to Prof. Chao Li for inspiring conversations. He also thianks Dr. Jian Wang and Dr. Liman Chen for helpful suggestions. The author is partially supported by National Key R\&D Program of China with grant no. 2020YFA0712800 and 2023YFA1009900 as well as the start-up fund from Westlake University.

\section{The proof}
\begin{proof}[Proof of Lemma \ref{Lem: concave}]
    By enlarging the region $K$ we may assume it to be a smooth connected and bounded region such that every component of $M\setminus K$ is unbounded. In the following, we just focus on one fixed component of $M\setminus K$, denoted by $E$. Two things need to be handled in the search of a bounded and weakly mean-concave region $\hat\Omega$ containing $K$ using variational method, including the obstacle issue caused by the existence of inner boundary $\partial E$ and a possible loss of compactness due to the non-compactness of $E$.
    
The obstacle issue is overcome by inserting a separation band $V$, from which we determine certain geometric quantities for later use in setting the $\mu$-bubble problem.  Fix a bounded open neighborhood $U\subset M$ of $K$ and we take the separation band 
$V$ to be $(\overline U\setminus K)\cap E$. For convenience, we make the illustration of the separation band $V$ in Figure \ref{Fig: separation band}, where the boundary $\partial V$ can be divided into two parts $\partial E$ and $\partial U$ respectively. To clarify we point out that both $\partial E$ and $\partial U$ can be disconnected in its worst case.

 \begin{figure}[htbp]
\centering
\includegraphics[width=6cm]{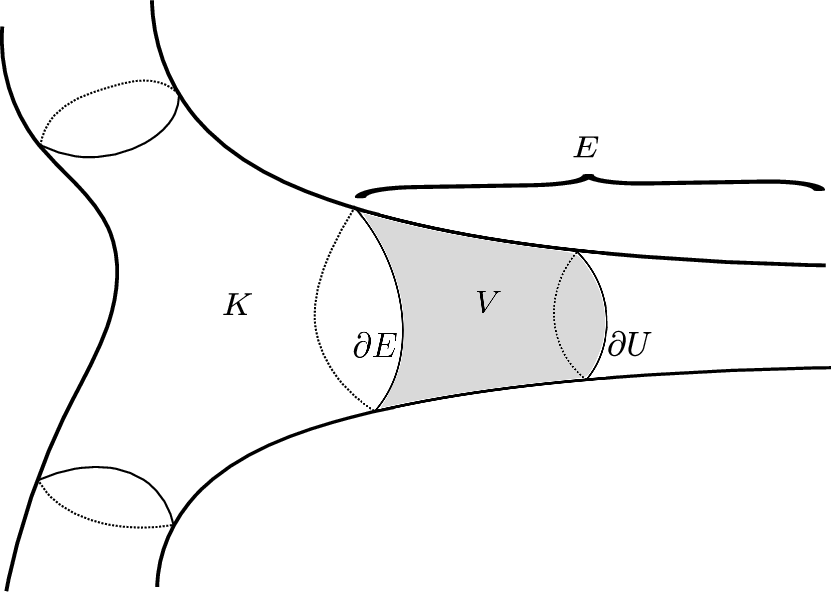}
\caption{The separation band $V$}
\label{Fig: separation band}
\end{figure}

Let us introduce two positive constants associated to the separation band $V$. The first constant is designed to describe the infimum area that an arbitrary homologically non-trivial hypersurface can have. For the definition we denote $\mathcal C_V$ to be the collection
$$
\left\{
\begin{array}{c}
	\mbox{integer multiplicity rectifiable $(n-1)$-currents $T$ with $\spt T\subset V$}\\\mbox{and $\partial T=0$ such that $T$ is homologically non-trivial}
\end{array}\right\}
$$
and take
$$
c_{gap}=\inf\left\{\mathbb M_g(T):T\in \mathcal C_V
\right\}.
$$
Here we recommend the audience to consult the book \cite{Leo83} for the precise definitions for current $T$ and mass $\mathbb M_g(T)$.  For our purpose, it is enough to consider currents as oriented hypersurfaces and mass as their areas. The second constant is used to measure the height of $V$ in the sense of area, which is defined by
$$
c_{height}=\inf\left\{\mathcal H^{n-1}_g(\Sigma)\left|
\begin{array}{c}
	\mbox{$\Sigma\subset V$ is a connected minimal hypersurface}\\\mbox{with $\partial\Sigma\subset \partial V$ intersecting both $\partial E$ and $\partial U$}
\end{array}\right.\right\}.
$$

Now let us verify the positivity of the constants $c_{gap}$ and $c_{height}$ as claimed.
To see that $c_1$ is positive, we take a conformal metric $\check g$ of $g$ such that
\begin{itemize}
\item $c^{-2}g\leq \check g\leq c^2g$ for some positive constant $c$;
\item and $(V,\check g)$ is a compact Riemannian manifold with convex boundary.
\end{itemize} 
From geometric measure theory we can find a smooth hypersurface $\check\Sigma$ which attains the least $\check g$-area among all the currents in $\mathcal C_V$. In particular, we have
$$
c_{gap}\geq c^{1-n}\cdot \mathcal H_{\check g}^{n-1}(\check\Sigma)
>0.$$
For the positivity of $c_2$ we take a smooth hypersurface $\Sigma_{sep}$ separating $\partial E$ and $\partial U$. Let $\Sigma\subset V$ be any connected minimal hypersurface with $\partial\Sigma\subset \partial V$ intersecting both $\partial E$ and $\partial U$. Clearly, $\Sigma$ must have non-empty intersection with $\Sigma_{sep}$ and so it follows from the monotonicity formula that 
$$
\mathcal H_g^{n-1}(\Sigma)\geq c_0(V,g,\dist(\Sigma_{sep},\partial V))>0,
$$
which yields $c_{height}>0$. 

At this stage we are ready to set the $\mu$-bubble problem to overcome the non-compactness issue and to find the desired bounded and weakly mean-concve region $\hat\Omega$. Recall that $(M,g)$ has sublinear volume growth. By definition we can find a sequence of positive constants $r_i\to+\infty$ such that for some point $p$ we have
$$
\frac{\vol_g(B_{r_i}(p))}{r_i}\to 0\mbox{ as }i\to\infty.
$$
Fix a smooth and proper function $\rho:E\to[0,+\infty)$ with $\rho^{-1}(0)=\partial E$ and $\Lip\rho\leq 2$ as well as $|\rho(\cdot)-\dist(\cdot,\partial E)|\leq 1$. Then it is not difficult to verify
$$
r_i^{-1}\vol_g\left(\rho^{-1}\left(\left[0,\frac{r_i}{2}\right]\right)\right)\to 0\mbox{ as }i\to\infty.
$$
From the co-area formula we have
$$
\int_{\frac{r_i}{4}}^{\frac{r_i}{2}}\mathcal H_g^{n-1}(\{\rho=\tau\})\,\mathrm d\tau=\int_{\{r_i/4\leq \rho\leq r_i/2\}}|\mathrm d\rho|_g\,\mathrm d\mathcal H^n_g\to 0\mbox{ as }i\to\infty.
$$
Then it follows that for each $i$ we can find a regular value $r_i/4<t_i<r_i/2$ of $\rho$ such that
$
\mathcal H^{n-1}_g(\{\rho=t_i\})\to 0\mbox{ as } i\to\infty.
$
In particular, we can fix $t_i$ large enough such that $V\subset\{\rho<t_i\}$ and
$$\mathcal H^{n-1}_g(\{\rho=t_i\})<\min\{c_{gap},c_{height}\}.$$
Just take another regular value $s_i>t_i$ of $\rho$ casually and denote $W$ to be $\rho^{-1}([0,s_i])$. In the following, we take $\eta:[0,s_i)\to(-\infty,0]$ to be a smooth function such that $\eta\leq 0$ everywhere, $\eta\equiv 0$ in $[0,t_i]$ and $\eta(s)\to-\infty$ as $s\to s_i$. The prescribed mean curvature function is now defined by $h=\eta\circ \rho$. 

Consider the class
$$
\mathcal C=\left\{
\begin{array}{c}
	\mbox{Caccioppoli sets $\Omega\subset M$ such that}\\\mbox{$E^c\subset \Omega$ and $\Omega\setminus E^c\Subset W$}
\end{array}\right\}
$$
and the functional
$$
\mathcal A^h(\Omega)=\mathcal H^{n-1}_g(\partial^*\Omega)-\int_{\Omega\setminus E^c}h\,\mathrm d\mathcal H^n_g,
$$
where $\partial^*\Omega$ is denoted to be the reduced boundary from \cite[Definition 3.3]{Giusti1977}.
Through a standard argument from geometric measure theory we can find a minimizer $\Omega^*\in \mathcal C$ such that
$$
A^h(\Omega^*)=\inf_{\Omega\in\mathcal C}\mathcal A^h(\Omega)
$$
and that $\Omega^*\cap \mathring E$ is a region in $\mathring E$ with smooth boundary $\partial\Omega^* \cap \mathring E$, where $\mathring E$ is denoted to be the interior of $E$.
Denote $U^*_1,\ldots,U^*_k$ to be the unbounded components of $M\setminus \Omega^*$ and $\hat\Omega_E$ to be the unique unbounded component of $M\setminus \left(\cup_iU^*_i\right)$ (the uniqueness comes from the connectedness of $K$). It is not difficult to verify $M\setminus \hat\Omega_E\subset E$. We claim that $\partial\hat\Omega_E$ is smooth and weakly mean-concave with respect to the outer unit normal as the boundary of $\hat\Omega_E$. Let $\hat\Sigma$ be an arbitrary component of $\partial\hat\Omega_E$. Notice that $\hat\Sigma$ is a common boundary of two unbounded connected regions $\hat\Omega_E$ and some $U^*_i$, then it has to be homologically non-trivial since we can construct a line intersecting $\hat\Sigma$ only once. Notice that the set $E^c \cup \rho^{-1}([0,t_i])$ also belongs to the class $\mathcal C$ and it follows from a direct comparison that
\begin{equation}\label{Eq: 1}
\mathcal H^{n-1}_g(\hat\Sigma)\leq\mathcal A^h(\Omega^*)\leq \mathcal H^{n-1}_g(\{\rho=t_i\})<\min\{c_{gap}, c_{height}\}.
\end{equation}
In particular, by definition of $c_{gap}$ we conclude that $\hat\Sigma\cap(E\setminus V)$ is non-empty. From the first variation formula of $\mathcal A^h$ it follows that $\hat\Sigma\cap\mathring E$
as the boundary of $\Omega^*$ has mean curvature $h|_{\hat\Sigma\cap\mathring E}$ with respect to the outer unit normal. In particular, $\hat\Sigma\cap V$ is a minimal hypersurface. If $\hat\Sigma$ intersects with $\partial E$, then we can find a connected minimal hypersurface $\hat\Sigma_0$ among components of $\hat\Sigma\cap V$ satisfying $\partial\hat\Sigma_0\subset \partial V$ and that $\hat\Sigma$ intersects both $\partial_-$ and $\partial_+$. This implies $\mathcal H^{n-1}_g(\hat\Sigma)\geq \mathcal H^{n-1}_g(\hat\Sigma_0)\geq c_{height}$, which contradicts to \eqref{Eq: 1}. Therefore, $\hat\Sigma$ does not touch $\partial E$ and so it is smooth everywhere. The mean-concavity of $\hat\Sigma$ comes directly from the non-positivity of the function $h$ after realizing that the outer unit normals of $\hat\Sigma$ with respect to $\Omega^*$ and $\hat\Omega$ coincide.

Finally let us take all unbounded components of $M\setminus K$ into consideration. After labeling them as $E_1,E_2,\ldots,E_l$ we can find regions $\hat\Omega_{E_1},\ldots,\hat\Omega_{E_l}$ with smooth weakly mean-concave boundary from above discussion. The desired bounded region is given by
$$
\hat\Omega=\bigcap_{i=1}^l\hat\Omega_i,
$$
which is obviously a bounded region containing $K$ with weakly mean-concave boundary.
\end{proof}
To be self-contained we would like to provide an alternative proof for the Calabi-Yau theorem, which will be used in the proof of Proposition \ref{Prop: splitting}. Of course, due to the use of geometric measure theory the dimension is assumed to be no greater than seven.
\begin{theorem}[Calabi-Yau]\label{Thm: Calabi-Yau}
Let $(M^n,g)$, $n\leq 7$, be a complete Riemannian manifold with nonnegative Ricci curvature outside a compact subset. Then $(M,g)$ has sublinear volume growth if and only if $M$ is compact.
\end{theorem}
\begin{proof}[An alternative proof without using volume comparison]
    Denote $K$ to be a smooth compact subset such that $(M,g)$ has nonnegative Ricci curvature outside $K$. Suppose that $(M,g)$ is non-compact but has sublinear volume growth. Then it follows from Lemma \ref{Lem: concave} that there is a bounded smooth region $\hat\Omega$ containing $K$ with weakly mean-concave boundary. Notice that $\partial\hat\Omega$ is weakly mean-convex as the boundary of $M\setminus\hat \Omega$ with respect to the corresponding outer unit normal. Since $M$ is non-compact, there is at least one unbounded component of $M\setminus\hat \Omega$ denoted by $E$. After applying \cite[Theorem 2]{CK92} to the end $E$ we conclude that $E$ must be isometric to the Riemannian product $(\partial E,g|_{\partial E})\times [0,+\infty)$, which obviously has linear volume growth. This leads to a contradiction!
\end{proof}

In the following, let us establish the splitting result which is involved in the proof of our main theorem.
\begin{proof}[Proof of Proposition \ref{Prop: splitting}]
    Here we take a similar argument from the author's previous work \cite{Zhu23b} based on Gromov's $\mu$-bubble. Let $\rho:M_\infty \to [0,+\infty)$ be a smooth proper function satisfying $\rho^{-1}(0)=\partial M_\infty$ and $\Lip \rho<1$. From \cite[Lemma 2.3]{Zhu23b} we can construct a smooth function $h_{\epsilon}:[0,\frac{1}{n\epsilon})\to(-\infty,0]$ for any $0<\epsilon<1$ such that
    \begin{itemize}
    \item $h_\epsilon$ satisfies
    $$
    \frac{n}{n-1}h_\epsilon^2+2h_\epsilon'=n(n-1)\epsilon^2\mbox{ on }\left[\frac{1}{2n},\frac{1}{n\epsilon}\right)
    $$
    and there is a universal constant $C=C(n)$ such that
    $$
    \left|\frac{n}{n-1}h_\epsilon^2+2h_\epsilon\right|_{L^\infty}\leq C\epsilon^2.
    $$
    \item $h_\epsilon<0$ and
    $$
    \lim_{t\to\frac{1}{n\epsilon}}h_\epsilon(t)=-\infty.
    $$
    \item as $\epsilon\to 0$, $h_\epsilon$ converges smoothly to the zero function on any closed interval.
    \end{itemize}
    
    Now we are ready to set appropriate $\mu$-bubble problems. Take $\epsilon$ such that $\frac{1}{n\epsilon}$ appears to be a regular value of $\rho$ and denote $V_\epsilon=\rho^{-1}([0,\frac{1}{n\epsilon}])$. Consider the class
    $$
    \mathcal C_\epsilon=\{\mbox{Caccioppoli sets $\Omega$ such that $M_\infty\setminus\Omega\Subset V_\epsilon$}\}
    $$
    and the functional
    $$
    \mathcal A_\epsilon(\Omega)=\mathcal H^{n-1}_g(\partial^*\Omega)-\int_{M_\infty\setminus\Omega}h_\epsilon\circ \rho\,\mathrm d\mathcal H^n_g.
    $$
    It follows from \cite[Proposition 2.1]{Zhu21} or \cite[Proposition 12]{CL20+} that we can find a smooth minimizer $\Omega^*_\epsilon$ of the functional $\mathcal A_\epsilon$ among the class $\mathcal C_\epsilon$. From a direct comparison we see
    $$
    \mathcal H^{n-1}_g(\partial\Omega^*_\epsilon)\leq \mathcal H^{n-1}_g(\partial\Omega^*_\epsilon)-\int_{M_\infty\setminus \Omega^*_\epsilon}h_\epsilon\circ \rho\,\mathrm d\mathcal H^n_g\leq \mathcal H^{n-1}_g(\partial M_\infty).
    $$
    The first variation formula of $\mathcal A_\epsilon$ yields that the mean curvature of $\partial\Omega^*_\epsilon$ as the boundary of $\Omega^*_\epsilon$ with respect to the unit outer normal $\nu^*$ is
    $$H_{\partial\Omega^*_\epsilon}=-(h_\epsilon\circ \rho)|_{\partial\Omega^*_\epsilon}.$$
    From the second variation formula of $\mathcal A_\epsilon$ we have
    \begin{equation}\label{Eq: 2}
    \begin{split}
    \int_{\partial\Omega^*_\epsilon}|\nabla\phi|^2-\frac{1}{2}&\Big(R(g)-R(g|_{\partial\Omega^*_\epsilon})+|A|^2\\
    &\qquad\qquad+H^2-2\partial_{\nu^*}(h_\epsilon\circ \rho)\Big)\phi^2\,\mathrm d\mathcal H^{n-1}_g\geq 0
    \end{split}
    \end{equation}
    for every $\phi$ in $C_0^\infty(\partial\Omega^*_\epsilon)$. Using the facts $R(g)\geq 0$ and
    $$
    |A|^2+H^2\geq \frac{n}{n-1}(h_\epsilon\circ \rho)^2
    $$
    as well as $\partial_{\nu^*}(h_\epsilon\circ \rho)\leq |\mathrm dh_\epsilon|\circ \rho$, we can write \eqref{Eq: 2} as
    \begin{equation}
    \begin{split}
     \int_{\partial\Omega^*_\epsilon}|\nabla\phi|^2+&\frac{1}{2}R(g|_{\partial\Omega^*})\phi^2\,\mathrm d\mathcal H^{n-1}_g\\
     &\geq \int_{\partial\Omega^*_\epsilon} \left.\left(\left(\frac{n}{n-1}h_\epsilon^2+2h_\epsilon'\right)\circ \rho \right)\right|_{\partial\Omega^*_\epsilon}\phi^2\,\mathrm d\mathcal H^{n-1}_g.
     \end{split}
    \end{equation}
    We claim that there is at least one component $\Sigma^*_\epsilon$ of $\partial\Omega_\epsilon^*$ having non-empty intersection with the compact subset $K_o:=\rho^{-1}([0,\frac{1}{2n}])$. Otherwise, $\Omega^*_\epsilon$ has to be disjoint with $K_o$ and in particular we can take $V$ to be the component of $M_\infty\setminus\Omega^*_\epsilon$ containing $K_o$. Let $\Sigma^*_\epsilon$ be some common boundary component of $\Omega^*_\epsilon$ and $V$. It is easy to construct a ray $\gamma:\big([0,+\infty),0\big)\to (M_\infty,\partial M_\infty)$ intersecting only once with $\Sigma^*_\epsilon$, which implies that $\Sigma^*_\epsilon$ represents a non-trivial $\mathbb Q$-homology class. From our assumption $\Sigma^*_\epsilon$ cannot admit any metric with positive scalar curvature. On the other hand, since $\Sigma^*_\epsilon$ as part of $\partial\Omega^*_\epsilon$ is disjoint from $K_o$, we conclude 
    $$
     \int_{\Sigma^*_\epsilon}|\nabla\phi|^2+\frac{1}{2}R(g|_{\partial\Omega^*})\phi^2\,\mathrm d\mathcal H^{n-1}_g>0\mbox{ for any non-zero }\phi\in C_0^\infty(\hat\Sigma^*_\epsilon).
    $$
    In particular, the first eigenvalue of the conformal Laplacian of $\Sigma^*_\epsilon$ is positive and we can construct a conformal metric of $\Sigma^*_\epsilon$ with positive scalar curvature, which leads to a contradiction.
    
    Now we analyze the limiting behavior of $\Omega^*_\epsilon$ as $\epsilon\to 0$. Notice that the functional $\mathcal A_\epsilon$ converges to the area functional $\mathcal A$ as $\epsilon\to 0$. From geometric measure theory up to a subsequence $\Omega^*_\epsilon$ converges to a (possibly empty) Caccioppoli set $\Omega^*_0$ whose boundary is locally area-minimizing. On the other hand, all the hypersurfaces $\Sigma^*_\epsilon$ intersect with a fixed compact subset $K_o$ and they have a uniform area bound
    $$
    \mathcal H^{n-1}_g(\Sigma^*_\epsilon)\leq \mathcal H^{n-1}_g(\partial\Omega^*_\epsilon)\leq H^{n-1}_g(\partial M_\infty).
    $$
    Fixed a point $q^*_\epsilon$ in $\Sigma^*_\epsilon\cap K_o$, the curvature estimate \cite[Theorem 3.6]{ZZ20} yields that the pointed hypersurface $(\Sigma^*_\epsilon,q^*_\epsilon)$ converges smoothly to a pointed minimal hypersurface $(\Sigma^*_0,q^*_0)$ up to a subsequence, which is part of the boundary $\partial\Omega^*_0$. It follows from \cite[Proposition 3.2]{Zhu23b} (with the original condition having non-zero degree to $T^n$ replaced by non-existence of positive scalar curvature) that the limit hypersurface $\Sigma^*_0$ must have vanishing Ricci curvature. With the uniform area bound passing to the limit we see $\mathcal H^{n-1}_g(\Sigma^*_0)\leq \mathcal H^{n-1}_g(\partial M_\infty)$. Combined with the Calabi-Yau theorem (see Theorem \ref{Thm: Calabi-Yau} above) we conclude that the limit hypersurface $\hat\Sigma^*_0$ must be compact. As a consequence, $\Sigma^*_0$ represents a non-zero $\mathbb Q$-homology class and it is area-minimizing as a boundary component of $\Omega^*_0$. Now it follows from the foliation argument as in \cite[Proposition 3.4]{Zhu20} that $(M_\infty,g)$ splits into the Riemannian product $(\partial M_\infty,g|_{\partial M_\infty})\times [0,+\infty)$.
\end{proof}

Now we are ready to prove the main theorem.
\begin{proof}[Proof of Theorem \ref{Thm: main}]
    The theorem follows immediately from the Calabi-Yau theorem in dimension two, and so we only need to deal with the case when $3\leq n\leq 4$. In these cases we are going to deduce some contradiction by assuming that $(M,g)$ is non-compact but has sublinear volume growth.
    
    First let us find the bounded region to use Lemma \ref{Lem: concave}. Recall that $M$ has nonnegative scalar curvature outside a compact subset $K_0$. Without loss of generality we can assume $K_0$ to be connected. Since $M$ is aspherical at infinity, we are able to find smooth compact sets $K_{n-2}\supset \cdots\supset K_1\supset K_0$ such that
 $$
 ( \pi_i(M\setminus K_{n-i})\to \pi_i(M\setminus K_{n-i-1}))=0\mbox{ for all }2\leq i\leq n-1.
 $$
 In the same way we just assume all $K_i$ to be connected.
    
    By our assumption in the beginning, $(M,g)$ has sublinear volume growth. Applying Lemma \ref{Lem: concave} to the compact set $K_{n-2}$ we can find a larger bounded region $\Omega\supset K_{n-2}$ with weakly mean-concave boundary $\partial\Omega$ with respect to the outer unit normal. By adding bounded components of $M\setminus\Omega$ and then passing to the component of $U$ containing $K_{n-2}$, we can further assume $U$ to be connected.
    
    In order to apply Proposition \ref{Prop: vanishing}  we set $X_i=M\setminus K_{n-i}$ for $3\leq i\leq n$ and set $X_2=M\setminus \overline \Omega$. Now we need to verify the conditions listed in Proposition \ref{Prop: vanishing}. From our construction it is clear that $\pi_i(X_i)\to \pi_i(X_{i+1})$ is the zero map for all $2\leq i\leq n-1$, so it remains to show the injectivity of $H_{n-1}(X_2,\mathbb Q)\to H_{n-1}(X_n,\mathbb Q)$ and we consider the exact sequence
    $$
    H_n(X_n,X_2,\mathbb Q)\to H_{n-1}(X_2,\mathbb Q)\to H_{n-1}(X_n,\mathbb Q).
    $$    
It suffices to show $ H_n(X_n,X_2,\mathbb Q)=0$. To see this we notice that $\overline U\setminus K_0$ consists of non-compact manifolds with boundary, which satisfies the well-known fact $H_n(\overline U\setminus K_0,\partial U)=0$. From the excision we have
$$
H_n(X_n,X_2)=H_n(M\setminus K_0,M\setminus \overline U)=H_n(\overline U\setminus K_0,\partial U)=0.
$$
In particular, we have $H_n(X_n,X_2,\mathbb Q)=0$.

    Now denote $M_\infty$ to be $M\setminus \Omega$ and it follows from Proposition \ref{Prop: vanishing}  that any embedded hypersurface representing a non-zero $\mathbb Q$-homology class cannot admit any metric with positive scalar curvature. Then it follows from Proposition \ref{Prop: splitting} that $(M_\infty ,g)$ splits into the Riemannian product 
    $$(\partial M_\infty,g|_{\partial M_\infty})\times[0,+\infty),$$ which leads to a contradiction to the fact that $(M,g)$ has sublinear volume growth.
    \end{proof}

Finally let us prove Theorem \ref{Thm: topology}.
\begin{proof}[Proof of Theorem \ref{Thm: topology}]
    To show that $M$ is homeomorphic to $\mathbb R^3$ it suffices to prove that $M$ is simply-connected at infinity (see \cite{Sta72} for instance). That is, for any compact subset $K$ we can find a larger compact subset $\tilde K$ such that the inclusion map $i_*:\pi_1(M\setminus\tilde K)\to \pi_1(M\setminus K)$ is the zero map. Let us argue by contradiction and suppose that there is a compact subset $K$ such that any disk bounded by a loop $\gamma$ outside $K$ intersects $K$.
    
    The strategy is to show that we can find arbitrarily large bounded region which has spherical boundary. From the contractibility of $M$ we see that $M$ is non-compact. Since $(M,g)$ has sublinear volume growth, from Lemma \ref{Lem: concave} we can find a bounded smooth region $\Omega$ containing $K$ whose boundary is weakly mean-concave. Denote $E_1,\,E_2,\ldots,E_l$ to be the unbounded components of $M\setminus\Omega$. Set the $\mu$-bubble problem on each $E_i$ in the same way as in the proof of Proposition \ref{Prop: splitting}. Since $(E_i,g)$ has positive scalar curvature, for $\epsilon$ small enough we can find a smooth region $E_{i,\infty}$ from the $\mu$-bubble problem such that $E_i\setminus E_{i,\infty}$ is bounded and that all $\partial E_{i,\infty}$ satisfies
     $$
     \int_{\partial E_{i,\infty}}|\nabla\phi|^2+\frac{1}{2}R(g|_{\partial E_{i,\infty}})\phi^2\,\mathrm d\mathcal H^{2}_g>0
     $$    for any non-zero $\phi\in C_0^\infty(\hat\Sigma^*_\epsilon)$. Taking the test function $\phi\equiv 1$ on each component of $\partial E_{i,\infty}$ and using the Gauss-Bonnet formula we conclude that $\partial E_{i,\infty}$ consists of $2$-spheres. Now we take $\tilde K$ to be $M\setminus\cup_i E_{i,\infty}$.
     
     We claim that any loop $\gamma$ in $M\setminus\tilde K$ can shrink to a point in $M\setminus  K$. Recall that $M$ is contractible. So the loop $\gamma$ bounds a disk $D$ in $M$. After perturbation we may assume that $D$ is transversal to $\partial\tilde K$. Let us take the component $\Sigma$ of $D\setminus \tilde K$ containing $\gamma$. Then $\Sigma$ is simply the disk $D$ with finitely many disjoint sub-disks $D_j$ removed. Since $\partial D_j$ is contained in the spherical boundary $\partial\tilde K$, we can find disks $D_j^*\subset \partial \tilde K$ with $\partial D_j^*=\partial D_j$. As a consequence, the set
     $$
     \Sigma\cup\left(\bigcup_j D_j^*\right)
     $$
     provides a disk outside $K$ with boundary $\gamma$, which leads to a contradiction to our assumption in the beginning.
\end{proof}

\appendix
\section{Finite-volume complete manifolds with nonnegative scalar curvature}\label{sec: finite volume}
\begin{lemma}
	There is a complete metric $g$ on $\mathbb R^n$, $n\geq 3$, with nonnegative scalar curvature such that $(M,g)$ has finite volume.
\end{lemma}
\begin{proof}
	Let us consider a conformally flat metric $g=u^{\frac{4}{n-2}}g_{euc}$ with $u$ a smooth positive function on $\mathbb R^n$ to be determined. To ensure $(\mathbb R^n,g)$ having nonnegative scalar curvature we just need to guarantee $\Delta u\leq 0$ concerning the formula
	$$
	-\Delta u=c_nR(g)u^{\frac{n+2}{n-2}}.
	$$
	
	The desired function $u$ is constructed as follows. We start with a function
	$$v=\left(r\ln r\right)^{-\frac{n-2}{2}}.$$
	A straight-forward computation gives
	$$
	\Delta v=-\frac{n-2}{2}\left(r\ln r\right)^{-\frac{n+2}{2}}\left(\frac{n-2}{2}\ln^2 r-\frac{n}{2}\right).
	$$
In particular, there is an absolute constant $r_0$ such that $\Delta v<0$ when $r\geq r_0$. Denote $v_0=v(r_0)$. To do composition we have to construct a function $\zeta:[0,+\infty)\to [0,v_0/2]$ satisfying
\begin{itemize}
	\item $\zeta(t)\equiv t$ in a neighborhood of $0$ and $\zeta(t)\equiv const.$ when $t\geq v_0/2$;
	\item $\zeta'(t)\geq 0$ and $\zeta''(t)\leq 0$ for all $t\geq 0$.
\end{itemize}
Such function can be constructed from integration. Take a nonnegative monotone-decreasing function $\eta:[0,+\infty)\to [0,1]$ such that $\eta\equiv 1$ in $[0,v_0/4]$ and $\eta\equiv 0$ in $[v_0/2,+\infty)$. It suffices to define
$$
\zeta(t)=\int_0^t\eta(s)\,\mathrm ds.
$$
Let $u=\zeta\circ v$. Note that $u$ is defined on the whole $\mathbb R^n$ since it is constant in the $r_0$-ball. It is direct to compute
$$
\Delta u=\zeta''|\nabla v|^2+\zeta'\Delta v.
$$
When $r\geq r_0$ it follows from $\Delta v<0$ and the construction of $\zeta$ that $\Delta u\leq 0$. When $r\leq r_0$ we simply have $\Delta u\equiv 0$ due to its constancy.

It remains to verify the completeness and the finite volume of $(\mathbb R^n,g)$. To see the completeness we compute
\begin{equation*}
	\begin{split}
		\dist(O,\infty)&=\int_{0}^{+\infty}u^{\frac{2}{n-2}}\,\mathrm dr\\
		&\geq \int_{r_0}^{+\infty}\frac{1}{r\ln r}\,\mathrm dr=\left.\ln\ln r\right|_{r_0}^{+\infty}=+\infty.
	\end{split}
\end{equation*}
On the other hand, the volume can be computed as
\begin{equation*}
	\begin{split}
		\vol(\mathbb R^n,g)&=\int_{\mathbb R^n}u^{\frac{2n}{n-2}}\,\mathrm dx\\
		&\leq \omega_n r_0^n\left(\frac{v_0}{2}\right)^{\frac{2n}{n-2}}+n\omega_n\int_{r_0}^{+\infty}\frac{1}{r\ln^n r}\,\mathrm dr\\
		&<+\infty.
	\end{split}
\end{equation*}
This completes the proof.
\end{proof}

\section{$\mathbb Q$-homology vanishing theorem}
The $\mathbb Q$-homology vanishing conjecture was raised by Gromov \cite[page 96]{Gromov23} as following
\begin{conjecture}\label{Conj: Q vanishing}
    Let $M^n$ be a closed manifold admitting positive scalar curvature. For any continuous map $f:M\to X$ mapping $M$ into an aspherical topological space $X$, we have $f_*([M])=0$ in $H_n(X,\mathbb Q)$.
\end{conjecture}
Based on the work \cite{CL20,LM23} the author and his collaborator \cite{HZ23+} proved for $3\leq n\leq 5$ that if $(M,g)$ is a closed $(n-1)$-manifold with positive scalar curvature and $X$ is an aspherical manifold, then for any continuous map $f:M\to X$ we have $f_*([M])=0\in H_{n-1}(X,\mathbb Q)$. In this paper, we need to use the following variant.
\begin{proposition}\label{Prop: vanishing}
Let $n=3$ or $4$, and $X^n$ be an $n$-manifold associated with a finite open exhaustion $X_2\subset X_3\subset \cdots \subset X_n=X$ satisfying 
\begin{itemize}
\item $H_{n-1}(X_2,\mathbb Q)\to H_{n-1}(X,\mathbb Q)$ is injective;
\item $\pi_i(X_i)\to \pi_{i}(X_{i+1})$ is the zero map for all $2	\leq i\leq n-1$.
\end{itemize} 
Assume that $(M^{n-1},g)$ is a closed manifold with positive scalar curvature. Then for any continuous map $f:M\to X_2$ we must have $f_*([M])=0\in H_{n-1}(X_2,\mathbb Q)$.
\end{proposition}
\begin{proof}
When $n=3$, $M$ can only be a topological $2$-sphere or the projective space $\mathbb RP^2$. Since $\pi_2(X_2)\to \pi_2(X)$ is the zero map, we obtain $f_*([M])=0\in H_2(M,\mathbb Q)$ in both cases and the conclusion comes from the injectivity of the map $H_{2}(X_2,\mathbb Q)\to H_{2}(X,\mathbb Q)$. When $n=4$, the argument is similar but slightly more complicated. 
By lifting we may assume $M$ to be orientable and so it follows from the classification result of orientable closed $3$-manifolds with positive scalar curvature that $M$ is diffeomorphic to the connected sum
$$
(\mathbb S^3/\Gamma_1)\#\cdots\#(\mathbb S^3/\Gamma_k)\#l(\mathbb S^2\times \mathbb S^1).
$$
Passing to some finite cover $\tilde M$ we see that $\tilde M$ is a connected sum of finitely many $\mathbb S^2\times \mathbb S^1$s. Since $\pi_2(X_2)\to \pi_2(X_3)$ is the zero map, we can break $\tilde M$ into spherical $3$-cycles in $X_3$. From the zero map $\pi_3(X_3)\to \pi_3(X)$ we conclude $\tilde f_*([\tilde M])=0\in H_3(X,\mathbb Q)$, where $\tilde f$ is denoted to be the composition of the map $f$ and the covering map $\tilde M\to M$. Once again, from the injectivity of the map $H_{3}(X_2,\mathbb Q)\to H_{3}(X,\mathbb Q)$ we obtain $\tilde f_*([\tilde M])=0\in H_3(X_2,\mathbb Q)$ and so $ f_*([ M])=0\in H_3(X_2,\mathbb Q)$.
\end{proof}

\end{document}